\documentclass[12pt,reqno]{amsart}
\usepackage{amssymb}
\usepackage{mathrsfs}
\usepackage{a4}
\usepackage{amscd}
\usepackage{tikz}
\usepackage{hyperref}
\usetikzlibrary{arrows}
\usepackage{accents}

\newcommand{\heute}{23 December 2015}

\theoremstyle{plain}
\newtheorem{theorem}{Theorem}[section]

\newtheorem{lemma}[theorem]{Lemma}
\newtheorem{corollary}[theorem]{Corollary}
\newtheorem{proposition}[theorem]{Proposition}

\theoremstyle{remark}

\newtheorem*{notn}{Notation}

\newcommand{\enref}[1]{\textup{(\ref{enum:#1})}}
\newcommand{\dashTwo}[1]{\textup{(\ref{two}${}'$)}}

\newcommand{\ignore}[1]{}

\newcommand{\f}[1][p]{\mathbb{F}_{#1}}

\newcommand{\Gro}[1]{Gr\"ob\-ner}

\newcommand{\Id}{\operatorname{Id}}

\DeclareMathOperator{\Bild}{Im}

\DeclareMathOperator{\rad}{rad}
\DeclareMathOperator{\soc}{soc}

\DeclareSymbolFont{bbold}{U}{bbold}{m}{n}
\DeclareSymbolFontAlphabet{\mathbbold}{bbold}

\newcommand{\stmod}[1]{\operatorname{stmod}(#1)}

\begin{document}

\title{The quaternion group has ghost number three}
\thanks{The first author was supported by  the Scientific and Technical Research Council
of Turkey (T\"UB\.ITAK-BIDEP-2219)}
\author[F. Altunbulak Aksu]{Fatma Altunbulak Aksu}
\email{altunbulak@cankaya.edu.tr}
\address{Dept of Mathematics and Computer Science \\ {\c C}ankaya University \\ Ankara \\ Turkey}
\author[D.~J. Green]{David J. Green}
\email{david.green@uni-jena.de}
\address{Institut f\"ur Mathematik \\ Friedrich-Schiller-Universit\"at Jena \\ 07737 Jena \\ Germany}
\date{\heute}

\begin{abstract}
We prove that the group algebra of the quaternion group $Q_8$ over any field of characteristic two has ghost number three.
\end{abstract}

\maketitle

\section{Introduction}
\noindent
The study of ghost maps in stable categories originated with Freyd's generating hypothesis in homotopy theory~\cite{Freyd:generating}, which is still an open question. In this paper we are concerned with ghosts in modular representation theory.
Let $G$ be a group and $K$ a field of characteristic $p$. A map $f \colon M \rightarrow N$ in the stable category $\stmod{KG}$ of finitely generated $KG$-modules is called a \emph{ghost} if it vanishes under Tate cohomology, that is if $f_* \colon \hat{H}^*(G,M) \rightarrow \hat{H}^*(G,N)$ is zero. The ghost maps then form an ideal in $\stmod{KG}$, and Chebolu, Christensen and Min\'{a}\v{c}~\cite{CheboluChristensenMinac:ghosts} define the \emph{ghost number} of $KG$ to be the nilpotency degree of this ideal.

Determining the exact value of the ghost number is hard in all but the simplest cases. In \cite{ChristensenWang:ghostNumbers},  Christensen and Wang studied ghost numbers for $p$-group algebras. They gave conjectural upper and lower bounds for the ghost number of an arbitrary $p$-group, and also showed that the ghost number (over a field of characteristic two) of the quaternion group $Q_8$ is either three or four. In our earlier paper~\cite{ghostnumber}, we established most cases of their conjectural bounds. In this paper, we shall prove the following theorem.

\begin{theorem}
\label{thm:main}
Let $K$ be any field of characteristic two. Then the group algebra $KQ_8$ has ghost number three.
\end{theorem}

\noindent
We claim therefore that every threefold ghost map $M \xrightarrow{f} N$ is stably trivial. To show this, we take any embedding $M \rightarrowtail I$ of $M$~in a finitely generated $KQ_8$-module and show that $f$ factors through~$I$.
\medskip

\noindent
In Section~\ref{sect:Dade}, we recall Dade's presentation of the group algebra $KQ_8$ and derive some properties of ghost maps, including the crucial Lemma~\ref{lem:15Jul2015}\@. In Section~\ref{sect:Kronecker}, we recall a theorem of Kronecker which classifies the linear relations on a vector space~$V$. This leads us to the construction of the lift in Section~\ref{sect:recoating}: We have $I = KQ_8 \otimes_K V$ for some $K$-vector space $V$. As we may assume $M$ to be projective-free, we have $M \subseteq J \otimes_K V$ for $J$ the Jacobson radical $J = J(KQ_8)$. Since a threefold ghost kills $\soc^3(M)$, it follows that $f$ factors through $M / \soc^3(M)$, which is a subspace of $(J/J^2) \otimes_K V \cong V^2$. That is, $M/\soc^3(M)$ is a linear relation on~$V$; and using Lemma~\ref{lem:15Jul2015} we are able to construct a lift for each  indecomposable summand in its Kronecker decomposition, thus proving the theorem.
\medskip

\noindent
\textbf{Acknowledgements}
The first author would like to thank the Institute for Mathematics of the University of Jena for their hospitality.

\section{Ghost maps and Dade's generators}
\label{sect:Dade}

\noindent
We only need the following property of ghost maps.

\begin{lemma}[\cite{CheboluChristensenMinac:ghosts}, Proposition~2.1]
\label{lem:DadeTateQ8}
Let $G$ be a $p$-group, $K$ a field of characteristic~$p$, and $M \xrightarrow{f} N$ a ghost map between projective-free $KG$-modules. Then $\Bild(f) \subseteq \rad(N)$ and $\soc(M) \subseteq \ker(f)$. \qed
\end{lemma}

\noindent
The next result is presumably well-known.

\begin{lemma}
\label{lem:finiteExtension}
Let $G$ be a finite group, $K/k$ a finite field extension, and $M \xrightarrow{f} N$ a map in $\stmod{kG}$. If $K \otimes_k M \xrightarrow{\Id_K \otimes f} K \otimes_k N$ is trivial in $\stmod{KG}$, then $f$ is trivial in $\stmod{kG}$.
Hence $\operatorname{ghost\ number}(kG) \leq \operatorname{ghost\ number}(KG)$.
\end{lemma}

\begin{proof}
As a map of $k$-vector spaces, inclusion $k \stackrel{i}{\hookrightarrow} K$ is a split monomorphism; let $K \stackrel{\pi}{\twoheadrightarrow} k$ be a splitting. Suppose that $\Id_K \otimes f$ factors through a finitely generated $KG$-projective module $P$. Then $f = (\pi \otimes \Id_N) \circ (\Id_K \otimes f) \circ (i \otimes \Id_M)$ also factors through~$P$, which is also a finitely generated $kG$-projective module. The last part follows, since extending scalars preserves ghost maps.
\end{proof}

\noindent
Consider now the quaternion group $Q_8 = \langle i, j \rangle$.
Let $K$ be a field of characteristic~$2$ which contains $\f[4] = \{0,1,\omega,\bar{\omega}\}$. In~\cite[(1.2)]{Dade:quaternion}, Dade defines $x,y \in J(KQ_8)$ by
\begin{xalignat*}{2}
x & = \omega i + \bar{\omega} j + ij & y & = \bar{\omega} i + \omega j + ij \, .
\end{xalignat*}
He then shows that $KQ_8$ is the $K$-algebra generated by $x,y$ with relations
\begin{xalignat*}{3}
x^2 &= yxy & y^2 &= xyx & xy^2 & = y^2x = x^2 y = yx^2 = 0 \, .
\end{xalignat*}
Hence $1,x,y,xy,yx,xyx,yxy,xyxy=yxyx$ is a $K$-basis of~$KQ_8$.

\begin{notn}
From now on, we write $R = KQ_8$ and $J = J(R) = \rad(R) = (x,y) \trianglelefteq R$.
\end{notn}

\begin{lemma}
\label{lem:surj}
Suppose that $[t + J^2(R)] \in \mathbb{P}(J/J^2)$ is neither $[x+J^2]$ nor $[y+J^2]$. Then for all $R$-modules $M$, the map $\rad(M) \rightarrow \rad^2(M)$, $m \mapsto tm$ is surjective.
\end{lemma}

\begin{proof}
It is enough to prove the case $M=R$; and by Nakayama it suffices to prove that the map $J / J^2 \rightarrow J^2/J^3$, $r + J^2 \mapsto tr + J^3$ is surjective. As $J/J^2$ and $J^2/J^3$ are both two-dimensional, $r \mapsto tr$ is surjective if and only if it is injective.

If $t \in \alpha x + \beta y + J^2(R)$ and $r \in \lambda x + \mu y + J^2(R)$ and then $tr \in \alpha \mu xy + \beta \lambda yx + J^3(R)$. So if $tr \in J^3$ then $\alpha \mu = 0 = \beta \lambda$. But the assumption on~$t$ means that $\alpha,\beta$ are both non-zero: so $r \in 0 + J^2$.
\end{proof}

\begin{lemma}
\label{lem:15Jul2015}
Suppose that $M \xrightarrow{f} N$ is a threefold ghost for $KQ_8$, with $M,N$ projective-free. Embedding $M$ in an injective module $R \otimes_K V$ for some $K$-vector space~$V$, we have $M \subseteq J \otimes_K V$.
Suppose further that $m \in M$ satisfies $m \in t \otimes v + J^2 \otimes_K V$ with $v \in V$ and $t \in \{x,y\}$. Then there is an $n \in N$ such that
\[
f(m) = \begin{cases} xyxn & t=x \\ yxyn & t=y \end{cases} \, .
\]
\end{lemma}

\begin{proof}
We treat the case $t=x$; the other case is analogous.
Hence $m = x \otimes v + xy u + yx w$ for some $u,w \in R \otimes_K V$, and so $yx m = xyxy w \in \soc(M)$.
Let
\[
M = N_0 \xrightarrow{f_1} N_1 \xrightarrow{f_2} N_2 \xrightarrow{f_3} N_3 = N
\]
be a realisation of $f$ as a threefold ghost, with $N_1$~and $N_2$ projective-free.
Recall from Lemma~\ref{lem:DadeTateQ8} that $\soc(N_{i-1}) \subseteq \ker(f_i)$ and $\Bild(f_i) \subseteq \rad(N_i)$.

Since $\soc(M) \subseteq \ker(f_1)$ it follows that $yxf_1(m) = 0$. As $\Bild(f_1) \subseteq \rad(N_1)$ there are $\alpha,\beta \in N_1$ with $f_1(m) = x \alpha + y \beta$.
Since $yx f_1(m) = 0$, we deduce that $yxy \beta = 0$ and hence $xy \beta \in \soc(N_1) \subseteq \ker(f_2)$.

Therefore $xy f_2(\beta) = 0$.
But $\Bild(f_2) \subseteq \rad(N_2)$, and so $f_2(\beta) = x\gamma + y\delta$ with $\gamma,\delta \in N_2$. From $xy f_2(\beta) = 0$ it follows that $xyx \gamma = 0$, hence $yx \gamma \in \soc(N_2) \subseteq \ker(f_3)$ and $yx f_3(\gamma) = 0$. It follows that
\[
f(m) = x f_3 f_2(\alpha) + yx f_3(\gamma) + xyx f_3(\delta) = x f_3 f_2(\alpha) \, ,
\]
since $f_3(\delta) \in \rad(N)$ and therefore $xyx f_3(\delta) \in \rad^4(N) = 0$. So $f(m) = x n'$ for $n' = f_3 f_2(\alpha) \in \rad^2(N)$. But then $n' = xy n'_1 + yx n'_2$ for some $n'_1,n'_2 \in N$, and so $f(m) = xyx n'_2$.
\end{proof}

\section{Kronecker's Theorem}
\label{sect:Kronecker}

\begin{theorem}[Kronecker]
\label{thm:Kronecker}
Let $K$ be a field, $V$ a finite-dimensional $K$-vector space, and $L \subseteq V^2$ a subspace. Suppose further that the pair $(V,L)$ is indecomposable, in the following sense: $V \neq 0$, and there is no proper direct sum decomposition $V = V_1 \oplus V_2$ such that $L = (L \cap V_1^2) \oplus (L \cap V_2^2)$.
Then there is a basis $e_1,\ldots,e_n$ of $V$ such that one of the following cases holds:
\begin{enumerate}
\item \label{enum:Kronecker-1}
$L$ has basis $(e_1,0), (e_2,e_1), (e_3,e_2), \ldots, (e_n,e_{n-1}),(0,e_n)$.
\item \label{enum:Kronecker-2}
$L$ either has basis $(e_1,0)$, $(e_2,e_1)$, $(e_3,e_2)$, \dots, $(e_n,e_{n-1})$ or it has basis $(0,e_1)$, $(e_1,e_2)$, $(e_2,e_3)$, \dots, $(e_{n-1},e_n)$.
\item \label{enum:Kronecker-3}
$L$ has basis $(e_2,e_1), (e_3,e_2), \ldots, (e_n,e_{n-1})$.
\item \label{enum:Kronecker-4}
$L = \{(v,F(v)) \mid v \in V \}$ for an automorphism $F$~of $V$ which has indecomposable rational canonical form with respect to the basis $e_1,\ldots,e_n$. A rational canonical form is indecomposable if it consists of only one block, whose characteristic polynomial is moreover a power of an irreducible element of $K[X]$.
\end{enumerate}
\end{theorem}

\begin{proof}
In the language of \cite[p.~ 112]{Benson:I}, the assumptions say that $L$~is an indecomposable linear relation on~$V$, which is the same thing as an indecomposable representation of the Kronecker quiver with $\ker(a) \cap \ker(b) \neq 0$. So the result can be read off from Kronecker's Theorem (Theorem~4.3.2 of~\cite{Benson:I}): note that Case~(i) in~\cite{Benson:I} corresponds to our cases \enref{Kronecker-2}~and \enref{Kronecker-4}\@.
\end{proof}

\begin{corollary}
\label{cor:Kronecker}
For every subspace $L \subseteq V^2$ there is a direct sum decomposition $V = \bigoplus_{i=1}^r V_i$ such that
\begin{enumerate}
\item $L = \bigoplus_{i=1}^r L_i$ for $L_i = L \cap V_i^2$.
\item For each $1 \leq i \leq r$ the pair $(V_i,L_i)$ is indecomposable in the sense of Theorem~\ref{thm:Kronecker}\@.
\end{enumerate}
We write $(V,L) = \bigoplus_{i=1}^r (V_i,L_i)$. \qed
\end{corollary}

\section{Constructing the lift}
\label{sect:recoating}
\noindent
Recall that $x+J^2$, $y+J^2$ is a basis of $J/J^2$.
Let $V$ be a finite dimensional $K$-vector space. Then any submodule $M \subseteq J \otimes_K V$ defines a subspace of $V^2$:
\begin{align*}
L_{x,y}(M) & := \{ (u,w) \in V^2 \mid x \otimes u + y \otimes w \in M + J^2 \otimes_K V \} \, .
\end{align*}
The proof of the following result is then immediate.

\begin{lemma}
\label{lemma:newpeel}
Let $M \subseteq J \otimes_K V$. Then
\begin{enumerate}
\item \label{enum:newpeel-1}
$\soc^3(M) = M \cap (J^2 \otimes_K V)$.
\item \label{enum:newpeel-4}
Set $L = L_{x,y}(M)$, and let $(V,L) = \bigoplus_{i=1}^r (V_i,L_i)$ be the direct sum decomposition of Corollary~\ref{cor:Kronecker}\@. If each $L_i$ has basis $(u_{i1},w_{i1}), \ldots,(u_{i,d_i}, w_{i,d_i})$, then for any choice of elements
\[
m_{ij} \in M \cap (x \otimes u_{ij} + y \otimes w_{ij} + J^2 \otimes_K V) \, .
\]
we have $M = \soc^3(M) + \sum_{i=1}^N M_i$, where $M_i = \sum_{j=1}^{d_i} R m_{ij}$. \qed
\end{enumerate}
\end{lemma}

\begin{proposition}
\label{prop:recoating}
For $M \subseteq J \otimes_K V$ set $L = L_{x,y}(M)$. Let $(V,L) = \bigoplus_{i=1}^r (V_i,L_i)$ be a decomposition into indecomposables. Suppose additionally that for each indecomposable pair $(V_i,L_i)$ which satisfies Case~\enref{Kronecker-4} of Theorem~\ref{thm:Kronecker}, the roots of the characteristic polynomial of the automorphism~$F$ all lie in~$K$.

Suppose further that $N$ is projective-free. Then every threefold ghost $M \xrightarrow{f} N$ extends to a map $R \otimes_K V \xrightarrow{\bar{f}} \rad^2(N)$.
\end{proposition}

\begin{proof}
Suppose first that the indecomposable $(V_i,L_i)$ satisfies Case~\enref{Kronecker-1} of Theorem~\ref{thm:Kronecker}\@.
Then $V_i$ has a basis $e_1,\ldots,e_n$ such that $L_i$ has basis $(0,e_1)$, $(e_1,e_2)$, $(e_2,e_3)$, \dots, $(e_{n-1},e_n), (e_n,0)$. By construction of~$L$, there are $m_0,\ldots,m_n \in M$ such that $m_j \in x \otimes e_j + y \otimes e_{j+1} + J^2 \otimes_K V$, where $e_0 = e_{n+1} = 0$. Since $\Bild(f) \subseteq \rad^3(N)$ there are $a_j,b_j \in N$ for $0 \leq j \leq n$ such that
\[
f(m_j) = xyx a_j + yxy b_j \, ;
\]
and by Lemma~\ref{lem:15Jul2015}, we may take $a_0 = b_n = 0$. We then define $\bar{f}$ on $R \otimes_K V_i$ by $\bar{f}(1 \otimes e_j) = xy b_{j-1} + yx a_j$. It follows that
\begin{xalignat*}{3}
\bar{f}(y \otimes e_1) & = f(m_0) & \bar{f}(x \otimes e_j + y \otimes e_{j+1}) & = f(m_j) & \bar{f}(x \otimes e_n) & = f(m_n)\, .
\end{xalignat*}
\item
The two subcases of Case~\enref{Kronecker-2} are analogous to each other, so we only consider the case where $L_i$ has basis $(0,e_1)$, $(e_1,e_2)$, $(e_2,e_3)$, \dots, $(e_{n-1},e_n)$. This corresponds to the case $f(m_n) = 0$ of Case~\enref{Kronecker-1} above, where we may take $a_n = 0$.
\medskip

\item
Case~\enref{Kronecker-3} is even simpler: this time we have $f(m_0) = f(m_n) = 0$ and therefore $b_0 = a_n = 0$.
\medskip

\item
Case~\enref{Kronecker-4}: By assumption, the matrix of $F$ with respect to the basis $e_1,\ldots,e_n$ of $V_i$ is a rational canonical form which has only one block, and the minimal polynomial of this block is $(X - \lambda)^n$ for some $\lambda \in K^{\times}$. It follows that there is a basis $e'_1,\ldots,e'_n$ of $V_i$ with respect to which the matrix of $F$ is the $(n \times n)$ Jordan block for the eigenvalue~$\lambda$. Consequently, $L_i$ has basis
\[
(e'_1,\lambda e'_1) \, , \qquad (e'_j, e'_{j-1} + \lambda e'_j) \; \text{for $2 \leq j \leq n$.}
\]
We may therefore pick elements $m_1,\ldots,m_n \in M$ such that
\begin{align*}
m_1 & \in (x + \lambda y) \otimes e'_1  + J^2 \otimes_K V
\\  m_j & \in y \otimes e'_{j-1} + (x + \lambda y) \otimes e'_j + J^2 \otimes_K V \quad \text{for $2 \leq j \leq n$.}
\end{align*}
So since $f(m_j) \in \rad^3(N)$ for all~$j$, and since $[x + \lambda y + J^2]$ is neither $[x+J^2]$ nor $[y+J^2]$, Lemma~\ref{lem:surj} tells us that we can inductively pick $\bar{f}(1 \otimes e'_1)$, \dots, $\bar{f}(1 \otimes e'_n) \in \rad^2(N)$ such that
\begin{align*}
\bar{f}((x + \lambda y) \otimes e'_1) & = f(m_1) \\
\bar{f}((x + \lambda y) \otimes e'_j) & = f(m_j) + \bar{f}(y \otimes e'_{j-1}) \quad \text{for $2 \leq j \leq n$.}
\end{align*}
Treating each summand $(V_i,L_i)$ in this way we obtain a map $\bar{f} \colon R \otimes_K V \rightarrow \rad^2(N)$, which therefore satisfies $\bar{f}(J^2 \otimes_K V) = 0$. It follows that all the equations above such as $\bar{f}(x \otimes e_j + y \otimes e_{j+1}) = f(m_j)$ can be simplified to $\bar{f}(m_j) = f(m_j)$. As $f$~and $\bar{f}$ are also both zero on $\soc^3(M) \subseteq J^2 \otimes_K V$, it follows by Lemma~\ref{lemma:newpeel} that $\bar{f} \vert_M = f$.
\end{proof}

\begin{proof}[Proof of Theorem~\ref{thm:main}]
By~\cite{ChristensenWang:ghostNumbers}, the ghost number is at least three. So we have to show that every threefold ghost $M \xrightarrow{f} N$ is stably trivial. Stripping projective summands if necessary, we may assume that $M,N$ are projective free. Taking an injective hull, we see that $M$ embeds in $R \otimes_K V$ for some finite-dimensional $K$-vector space $V$. Since $M$ is projective free, we actually have $M \subseteq J \otimes_K V$.

By Lemma~\ref{lem:finiteExtension}, we may replace $K$ by a finite extension field: so we may assume that $\f[4] \subseteq K$. Set $L = L_{x,y}(M)$. Corollary~\ref{cor:Kronecker} says that $(V,L)$ is a direct sum of indecomposables. Replacing $K$ by a finite extension field again if necessary, we may assume in Case~\enref{Kronecker-4} of Theorem~\ref{thm:Kronecker} that the characteristic polynomial of the automorphism $F$ always splits over~$K$. By Proposition~\ref{prop:recoating}, it follows that $f$ extends to a map $\bar{f} \colon R \otimes_K V \rightarrow \rad^2(N)$, meaning that $f$ is stably trivial.
\end{proof}

\ignore{
\bibliographystyle{abbrv}
\bibliography{united}

\begin{thebibliography}{1}

\bibitem{ghostnumber}
F.~Altunbulak~Aksu and D.~J. Green.
\newblock On the {C}hristensen--{W}ang bounds for the ghost number of a
  $p$-group algebra.
\newblock \emph{J. Group Theory}, accepted, 2015.
\newblock arXiv:1502.05727 [math.GR].

\bibitem{Benson:I}
D.~J. Benson.
\newblock {\em Representations and Cohomology. {I}}.
\newblock Cambridge Studies in Advanced Math., vol.~30. Cambridge University
  Press, Cambridge, second edition, 1998.

\bibitem{CheboluChristensenMinac:ghosts}
S.~K. Chebolu, J.~D. Christensen, and J.~Min{\'a}{\v{c}}.
\newblock Ghosts in modular representation theory.
\newblock {\em Adv. Math.}, 217(6):2782--2799, 2008.

\bibitem{ChristensenWang:ghostNumbers}
J.~D. Christensen and G.~Wang.
\newblock Ghost numbers of group algebras.
\newblock {\em Algebr. Represent. Theory}, 18(1):1--33, 2015.

\bibitem{Dade:quaternion}
E.~C. Dade.
\newblock Une extension de la th\'eorie de {H}all et {H}igman.
\newblock {\em J. Algebra}, 20:570--609, 1972.

\bibitem{Freyd:generating}
P.~Freyd.
\newblock Stable homotopy.
\newblock In {\em Proc. {C}onf. {C}ategorical {A}lgebra ({L}a {J}olla,
  {C}alif., 1965)}, pages 121--172. Springer, New York, 1966.

\end{thebibliography}
} 

\end{document}